\documentclass[a4paper,12pt]{article}
\usepackage{amsmath}
\usepackage{amsfonts}
\usepackage{appendix}
\usepackage{amsthm}
\usepackage{amssymb}
\usepackage{enumitem}
\usepackage{yhmath}
\usepackage{fullpage}
\usepackage[all]{xy}
\usepackage{marginnote}
\usepackage{pgf,tikz,pgfplots}
\pgfplotsset{compat=1.15}
\usepackage{mathrsfs}
\usepackage{color}
\usetikzlibrary{arrows}
\numberwithin{equation}{section}
\usepackage{graphicx}
\usepackage{caption}
\usepackage{subcaption}
\usepackage[english]{babel}
\usepackage{tocloft}
\usepackage[colorlinks, linkcolor=blue, citecolor=blue, urlcolor=blue]{hyperref}
\usepackage{latexsym}
\usepackage{mathrsfs}
\usepackage{pifont}

\newcommand{\C}{\mathbb{C}}
\newcommand{\R}{\mathbb{R}}
\newcommand{\Z}{\mathbb{Z}}
\newcommand{\N}{\mathbb{N}}

\newcommand{\la}{\lambda}

\newtheoremstyle{rem}{3pt}{3pt}{}{}
{\bfseries}{.}{.5em}{}

\newtheorem{theo}{Theorem}[section]
\newtheorem*{theo*}{Theorem}
\newtheorem{lemm}[theo]{Lemma}

\newtheorem{example}[theo]{Example}
\newtheorem{prop}[theo]{Proposition}

\newtheorem{coro}[theo]{Corollary}

\theoremstyle{rem}
\newtheorem{remax}[theo]{Remark}

\newenvironment{rema}
  {\pushQED{\qed}\remax}
  {\popQED\endremax}

\theoremstyle{definition}
\newtheorem{defi}[theo]{Definition}

\newtheorem*{term*}{Notation/Terminology}

\DeclareMathOperator{\Pol}{Pol}
\DeclareMathOperator{\Hom}{Hom}

\usepackage{geometry}
\newcommand{\RC}[3]{\operatorname{RC}^{#3}_{#1,#2}}

\title{Identities for Rankin--Cohen brackets, Racah coefficients and associativity}

\author{Q. Labriet\footnote{Department of Mathematics, Aarhus University, Ny Munkegade 118, DK-8000, Aarhus C, Denmark. \emph{email adress}: quentin.labriet@math.au.dk}, L. Poulain d'Andecy\footnote{Laboratoire de math\'ematiques de Reims LMR, UMR 9008, Universit\'e de Reims Champagne-Ardenne, Moulin de la Housse BP 1039, 51100 Reims, France. \emph{email adress}: loic.poulain-dandecy@univ-reims.fr}}

\date{}

\begin{document}
\maketitle

\begin{abstract}
We prove an infinite family of identities satisfied by the Rankin--Cohen brackets involving the Racah polynomials. A natural interpretation in the representation theory of $sl(2)$ is provided. From these identities and known properties of the Racah polynomials follows a short new proof of the associativity of the Eholzer product. Finally, we discuss, in the context of Rankin--Cohen algebras introduced by D.Zagier, how any algebraic identity satisfied by the Rankin--Cohen brackets can be seen as a consequence of the set of identities presented in this paper.   
\end{abstract}

\section{Introduction}

The Rankin--Cohen brackets (RC brackets, for short) were introduced as a family of differential operators which transform a pair of modular forms into a modular form of highest weight (\cite{Cohen75}, \cite{Rankin56}). More generally, we can consider them as bi-differential operators acting on pairs of holomorphic functions of a complex variable. For $\lambda_1,\lambda_2 \in \C$ and $n\in\mathbb{Z}_{\geq0}$, the RC bracket is defined via the following formula:
\begin{equation}\label{def:RC}
\RC{\lambda_1}{\lambda_2}{n}(f,g)(z)=\sum_{s=0}^n(-1)^s\binom{\lambda_1+n-1}{n-s}\binom{\lambda_2+n-1}{s} \partial^{(s)}f(z)\partial^{(n-s)}g(z).
\end{equation}
In \cite{Zagier94}, D.Zagier studied the RC brackets from an algebraic point of view and showed that they satisfy some algebraic identities which can be checked directly such as:
\begin{align}
0&=[[f_1,f_2]_1,f_3]_1+[[f_2,f_3]_1,f_1]_1+[[f_3,f_1]_1,f_2]_1;\label{eq:IdentityRankinCohen1}\\
0&=\lambda_3[[f_1,f_2]_1,f_3]_0+\lambda_1[[f_2,f_3]_1,f_1]_0+\lambda_2[[f_3,f_1]_1,f_2]_0. \label{eq:IdentityRankinCohen2}
\end{align}
Following \cite{Zagier94}, a standard notation $[f_i,f_j]_n$ is used instead of $\RC{\lambda_i}{\lambda_j}{n}(f_i,f_j)$. This notation is hiding the parameters of the brackets and the convention is the following: it is understood that each function $f_i$ is associated to a parameter $\lambda_i$, and moreover that the function $[f_i,f_j]_k$ is associated to the parameter $\lambda_i+\lambda_j+2k$.

W.Eholzer introduced the following product on the space of all modular forms which we call the \emph{Eholzer product} (see \cite{CMZ97,Zagier94}):
\begin{equation}\label{def:StarProduct}
f\star g =\sum_n~ [f,g]_n\hbar^n.
\end{equation}
He conjectured that this product is associative which gives many other algebraic identities satisfied by the RC brackets. The associativity was proved with two different methods in \cite{CMZ97} and \cite{ConnesMoscovici03}. More details on Rankin--Cohen brackets can be found for example in \cite{CMZ97,ConnesMoscovici03,KobayashiPevzner16, Medina16, Pevzner18, Unterberger96,Zagier94} and references therein.

The main result of this paper is the following set of identities satisfied by the RC brackets.
\begin{theo}\label{thm:EquationRC}
Let $\lambda_1,\lambda_2,\lambda_3>0$ and $n\in \N$. Then the following equality is satisfied:
\begin{equation}\label{eq:EquationRCbrackets}
[[f_1,f_2]_k,f_3]_{n-k}=\sum_{p=0}^nU^{\lambda_1,\lambda_2;k}_{\lambda_3;n,p}[f_1, [f_2,f_3]_p]_{n-p}\ ,
\end{equation}
where 
\begin{equation}\label{eq:FormulaCoeff}
U^{\lambda_1,\lambda_2;k}_{\lambda_3;n,p}=\binom{n}{k}\frac{(\lambda_2)_k(\lambda_3)_{n-k}(\lambda_1+\lambda_2+\lambda_3+n-1)_p}{(\lambda_3)_p(\lambda_2+\lambda_3+p-1)_p(\lambda_2+\lambda_3+2p)_{n-p}} R_{p,k}.
\end{equation}
Here $(x)_m$ is the Pochhammer symbol $
(x)(x+1)\dots (x+m-1)$ and $R_{p,k}$ denotes the values of the Racah polynomials (see \cite{KLS10}) given by
\[R_{p,k}={}_4F_3\left(\begin{array}{c} -p\ ,\ p+\lambda_2+\lambda_3-1\ ,\ -k\ ,\ k+\lambda_1+\lambda_2-1 \\ \lambda_2\ ,\ \lambda_1+\lambda_2+\lambda_3+n-1\ ,\ -n\end{array}; 1\right).\] 
\end{theo}
Using the full notation including the parameters, the equality \eqref{eq:EquationRCbrackets} reads as the following equality of operators: 
\begin{equation}\label{eq:EquationRC}
\RC{\lambda_1+\lambda_2+2k}{\lambda_3}{n-k}\circ (\RC{\lambda_1}{\lambda_2}{k}\otimes Id)=\sum_{p=0}^n U^{\lambda_1,\lambda_2;k}_{\lambda_3;n,p}\RC{\lambda_1}{\lambda_2+\lambda_3+2p}{n-p}\circ (Id\otimes \RC{\lambda_2}{\lambda_3}{p})\ .
\end{equation}
This result gives a large family of relations generalising the identities mentioned above. In fact, as a first application, we provide yet another proof of the associativity of the Eholzer product. Indeed, we show that the associativity becomes an elementary consequence of some classical properties of Racah polynomials. Thus the identities corresponding to the associativity of the Eholzer product are consequences of the identities proved in Theorem \ref{thm:EquationRC}. We will go a little bit further and in connections with the notion of Rankin--Cohen algebras introduced in \cite{Zagier94}, we will show that, in some sense, all algebraic identities satisfied by the RC brackets are consequences of the ones above (and of the well-known antisymmetry property).

\vskip .2cm
Now let us interpret Theorem \ref{thm:EquationRC} in terms of the representation theory of the Lie algebra $\mathfrak{sl}_2$. Under suitable conditions on $\lambda_1,\lambda_2\in\mathbb{C}$ (see Section \ref{sec:VermaModules} for more precision), we have the following decomposition of the tensor product of two highest weight Verma modules $V_{\lambda_1},V_{\lambda_2}$:
\begin{equation*}
V_{\lambda_1}\otimes V_{\lambda_2}\simeq \bigoplus_{k\in \N} V_{\lambda_1+\lambda_2+2k}\ .
\end{equation*}
For a suitable realisation of the $V_\lambda$'s as spaces of holomorphic functions on the upper half-plane, it is known (see \cite{KobayashiPevzner16} for example) that the RC bracket $\RC{\lambda_1}{\lambda_2}{k}$ is an intertwining operator between $V_{\lambda_1}\otimes V_{\lambda_2}$ and $V_{\lambda_1+\lambda_2+2k}$ generating the space $\Hom_{\mathfrak{sl}_2}(V_{\lambda_1}\otimes V_{\lambda_2},V_{\lambda_1+\lambda_2+2k})$. We give another proof of this result in Section \ref{sec:IntertwiningOperator}.

Then looking at the threefold tensor product of Verma modules we have the following decomposition:
\begin{equation*}
V_{\lambda_1}\otimes V_{\lambda_2}\otimes V_{\lambda_3}\simeq \bigoplus_{n\in \N}~ (n+1)V_{\lambda_1+\lambda_2+\lambda_3+2n}.
\end{equation*}
Each of the two families
$$\{\RC{\lambda_1+\lambda_2+2k}{\lambda_3}{n-k}\circ (\RC{\lambda_1}{\lambda_2}{k}\otimes Id)\}_{0\leq k\leq n}\quad\text{and}\quad\{\RC{\lambda_1}{\lambda_2+\lambda_3+2p}{n-p}\circ (Id\otimes \RC{\lambda_2}{\lambda_3}{p})\}_{0\leq p\leq n}\ $$ 
forms a basis of the space of intertwining operators $\Hom_{\mathfrak{sl}_2}(V_{\lambda_1}\otimes V_{\lambda_2}\otimes V_{\lambda_3},V_{\lambda_1+\lambda_2+\lambda_3+2n})$.
Hence, the identity \eqref{eq:EquationRC} can be understood as a change of basis between these two bases of the space of intertwining operators.

The fact that the transition coefficients are expressed in terms of Racah polynomials can be seen as a consequence of an action of the Racah algebra, realised as the commutator of $\mathcal{U}(\mathfrak{sl}_2)$ in $\mathcal{U}(\mathfrak{sl}_2)^{\otimes 3}$, on the space of intertwining operators from $V_{\lambda_1}\otimes V_{\lambda_2}\otimes V_{\lambda_3}$ to $V_{\lambda_1+\lambda_2+\lambda_3+2n}$. Or equivalently, it is a consequence of an action of the Racah algebra on the multiplicity spaces. We refer for example to \cite{LabrietPoulain23} where it is shown how this action of the Racah algebra leads to convolution identities for Jacobi polynomials, interpreted as a change of basis in the multiplicity spaces. Roughly speaking, taking the adjoint makes the RC brackets appear naturally and provides the identities presented in this paper.

\paragraph{Acknowledgements.} The authors warmly thank M.Pevzner and G.Zhang for their interest in this work. The first author is supported by a research grant from the Villum Foundation (Grant No. 00025373). The second author is supported by Agence Nationale de la Recherche Projet AHA ANR-18-CE40-0001 and the international research project AAPT of the CNRS.

\section{Tensor product of Verma modules of $\mathfrak{sl}_2$}\label{sec:VermaModules}
In this section, we recall some classical results for the tensor product of Verma modules of $\mathfrak{sl}_2$ and the associated intertwining operators. We gives proofs of some classical results using a realisation of Verma modules on the space of polynomials in one variable combined with some properties of classical Jacobi polynomials. 

\subsection{Lowest weight Verma modules}
We work over the complex field $\C$ and we denote the standard basis $H=\begin{pmatrix}
1&0\\0&-1
\end{pmatrix}$
, $F=\begin{pmatrix}
0&0\\1&0
\end{pmatrix}$ and $E=\begin{pmatrix}
0&1\\0&0
\end{pmatrix}$ of the Lie algebra $\mathfrak{sl}_2$. As usual we consider $H,E,F$ as generators of the universal enveloping algebra $\mathcal{U}(\mathfrak{sl}_2)$. The Casimir operator in $\mathcal{U}(\mathfrak{sl}_2)$ is defined as $C=\frac{1}{4}H^2+\frac{1}{2} H+FE$.

For $\lambda\in \C$, we realise the highest weight Verma module $V_\lambda$ with highest weight $-\lambda$ on the space of polynomials $\Pol(x)$ via the following formulas:
\begin{align*}
&\pi_\lambda(H)P=-\lambda P-2xP',\\
&\pi_\lambda(E)P=-P',\\
&\pi_\lambda(F)P=x^2P'+\lambda xP.
\end{align*}

These representations are known to be irreducible if $\lambda \notin  \Z_{\leq 0}$. Note that if $\lambda=-n \in \Z_{\leq 0}$ then these formulas define an irreducible finite-dimensional representation on the quotient $\Pol(x)/<x^n>$. If $\lambda\in\R_{>0}$ this representation is derived from a unitary representation of the universal cover of $SL_2(\R)$, and from a unitary representation of $SL_2(\R)$ for $\lambda\in \Z_{>0}$ \cite{HoweTan92}.

The contragredient representation $V_\lambda^*$ of $V_\lambda$ is realised on the dual $\Pol(x)^*$ via :
\begin{equation}
\pi_\lambda^*(Z)f(v)=f(-\pi_\lambda(Z)v), \text{ for all } Z\in \mathfrak{sl}_2.
\end{equation}
We identify the dual $\Pol(x)^*$ with $\Pol(x)$ via the Fischer inner product on $\Pol(x)$:
\begin{equation}\label{def:FischerProduct}
\langle P,Q\rangle_F=\left.P(\partial_x)\overline{Q(x)}\right|_{x=0} \quad\text{ or }\quad \langle x^n,x^m\rangle_F =n!\delta_{n,m}.
\end{equation}
and we can realise the contragredient representation on $\Pol(x)$ via the formula:
\begin{equation}
\pi_\lambda^*(Z)P=-\pi_\lambda(Z)^\dagger,
\end{equation}
where $^\dagger$ denotes the adjoint operator with respect to $\langle\cdot,\cdot\rangle_F$. Notice that the $\dagger$ operation exchanges $x$ and $\partial_x$. The module $V_\lambda^*$ is now a lowest weight module of weight $\lambda$ which is explicitly given by the formulas:
\begin{align*}
&\pi^*_\lambda(H)P=\lambda P+2xP',\\
&\pi^*_\lambda(E)P=xP,\\
&\pi^*_\lambda(F)P=-(xP''+\lambda P').
\end{align*}

Before considering the tensor product of two lowest weight Verma modules, we recall some facts about the family of Jacobi polynomials $\{P_\ell^{(\alpha,\beta)}\}$ which can be defined by the following expression \cite{KLS10}:
\begin{equation}
P^{(\alpha,\beta)}_\ell(v)=\frac{(\alpha+1)_\ell}{\ell !}{}_2F_1\left(\begin{array}{c} -\ell\ ,\ 1+\alpha+\beta+\ell\\ \alpha+1 \end{array};\frac{1-v}{2}\right).
\end{equation}
Equivalently they are given by
\begin{equation}
P^{(\alpha,\beta)}_\ell(v)=\frac{1}{2^\ell}\sum_{s=0}^\ell(-1)^s\binom{\ell+\alpha}{\ell-s}\binom{\ell+\beta}{s}(1-v)^s(1+v)^{\ell-s}.
\end{equation}
If $\alpha,\beta \notin \Z_{<0}$ and $\alpha+\beta +1 \notin \Z_{<0}$ then $P^{(\alpha,\beta)}_\ell$ is of degree $\ell$ and thus the family $\{P^{(\alpha,\beta)}_\ell\}$ is a basis of $\Pol(v)$. Define $(\cdot,\cdot)_J$ to be an inner product for which the Jacobi polynomials are orthogonal. For $\alpha,\beta>-1$ it is given by
\begin{equation}
(f,g)_J=\int_{-1}^1f(t)\overline{g(t)}(1-y)^{\alpha}(1+t)^\beta~dt.
\end{equation} 

The Jacobi polynomial $P_\ell^{(\alpha,\beta)}$ is an eigenvector with eigenvalue $-\ell(\ell+\alpha+\beta+1)$ for the Jacobi (or hypergeometric) differential operator
$\Phi^{(\alpha,\beta)}$ defined by:
\begin{equation}
\Phi^{(\alpha,\beta)}=(1-v^2)\partial_v^2+(\alpha-\beta-(\alpha+\beta+2)v)\partial_v.
\end{equation}

\subsection{Tensor product of two lowest weight Verma modules}

We fix $\lambda_1,\lambda_2\in \C\backslash \Z_{\leq 0}$ such that $\lambda_1+\lambda_2 \notin \Z_{\leq 0}$.

We consider the tensor product $V_{\lambda_1}^*\otimes V_{\lambda_2}^*$ of the two lowest weight Verma modules of weight $\lambda_1$ and $\lambda_2\in\C\backslash \Z_{\leq 0}$. Firstly, we use known facts about Jacobi polynomials to recover some well known results about this tensor product (see \cite{HoweTan92} for example). Then this will be used to easily compute intertwining operators between the tensor product and the irreducible components in its direct sum decomposition. 

The tensor product representation $V_{\lambda_1}^*\otimes V_{\lambda_2}^*$ is realised on $\Pol(x,y)$ via:
\begin{align}
&\pi_{\lambda_1}^*\otimes\pi_{\lambda_2}^*(H)P=(\lambda_1+\lambda_2)P+2(x\partial_x+y\partial_y)P,\\
&\pi_{\lambda_1}^*\otimes\pi_{\lambda_2}^*(E)P=(x+y)P\\
&\pi_{\lambda_1}^*\otimes\pi_{\lambda_2}^*(F)P=-\left((x\partial_x^2+y\partial^2_y)P+\lambda_1\partial_xP+\lambda_2\partial_y P\right).
\end{align} 
We now introduce another model in which the Casimir operator of $\mathcal{U}(\mathfrak{sl}_2)$ will have a familiar expression. For this, we use the change of coordinates $(x,y)\mapsto (x+y,\frac{y-x}{x+y})$, with inverse $(t,v)\mapsto (\frac{t(1-v)}{2},\frac{t(1+v)}{2})$. This change of coordinates has a geometrical interpretation in terms of symmetric cones when one considers the tensor product of two holomorphic discrete series representations of $SL_2(\R)$ (see \cite{Labriet22}).

Pulling back this change of coordinates to polynomials we get a map
\begin{equation}
\Psi:P\mapsto P\left(\frac{t(1-v)}{2},\frac{t(1+v)}{2}\right)\ ,
\end{equation}
and we realise the tensor product on the image of $\Psi$, which is simply $\Pol(t,tv)$. More precisely, the following lemma describes the image of $\Psi$ as a direct sum which will lead to the decomposition into irreducible representations of the tensor product.
\begin{lemm}\label{lem:ImagePsi}
For $\lambda_1,\lambda_2\in \C\backslash \Z_{\leq 0}$ such that $\lambda_1+\lambda_2 \notin \Z_{\leq 0}$, we have:
\begin{equation}
Im(\Psi)=\bigoplus_{\ell\geq 0} t^\ell \Pol(t)\otimes \C P_\ell^{(\lambda_1-1,\lambda_2-1)}(v).
\end{equation}
\end{lemm}
\begin{proof}
First, notice that if $\tilde{P}(x,y)$ is a homogeneous polynomial of degree $\ell$ then $\Psi(\tilde{P})(t,v)=\left(\frac{t}{2}\right)^\ell \tilde{P}(1-v,1+v)$. This remark applies to the polynomial 
\begin{equation}\label{eq:JacobiTwoVariables}
\tilde{P}^{(\lambda_1-1,\lambda_2-1)}_\ell(x,y)=\sum_{s=0}^\ell(-1)^s\binom{\ell+\lambda_1-1}{\ell-s}\binom{\ell+\lambda_2-1}{s}x^sy^{\ell-s},
\end{equation}
and shows that $\Psi(\tilde{P}^{(\lambda_1-1,\lambda_2-1)}_\ell)(t,v)=t^\ell P_\ell^{(\lambda_1-1,\lambda_2-1)}(v)$ where $P_\ell^{(\lambda_1-1,\lambda_2-1)}$ is the Jacobi polynomial. Hence, we have for any polynomial $Q$ in one variable:
\begin{equation*}
\Psi(Q(x+y)\tilde{P}_\ell^{(\lambda_1-1,\lambda_2-1)})=t^\ell Q(t)P_\ell^{(\lambda_1-1,\lambda_2-1)}(v),
\end{equation*}
which proves
\begin{equation*}
\sum_{\ell\geq 0} t^\ell \Pol(t)\otimes \C P_\ell^{(\lambda_1-1,\lambda_2-1)}(v)\subset Im(\Psi).
\end{equation*}
Moreover we have:
\begin{equation*}
\Psi(x^ny^m)=\left(\frac{t}{2}\right)^{n+m}(1-v)^n(1+v)^m.
\end{equation*}
Since the Jacobi polynomials are a basis for $\Pol(v)$ under our assumptions on $\lambda_1$ and $\lambda_2$, we can expand $(1-v)^n(1+v)^m$ in this basis which shows the opposite inclusion.
\end{proof}

In the coordinates $(t,v)$ one gets by direct calculation the following formulas for the representation:
\begin{align}\label{eq:ModeleVermaModule(t,v)}
&\pi_{\lambda_1}^*\otimes\pi_{\lambda_2}^*(H)P=(\lambda_1+\lambda_2)P+2t\partial_t P,\\
&\pi_{\lambda_1}^*\otimes\pi_{\lambda_2}^*(E)P=tP,\\
&\pi_{\lambda_1}^*\otimes\pi_{\lambda_2}^*(F)P=-\left(t\partial_t^2 P+(\lambda_1+\lambda_2)\partial_t P +t^{-1}\Phi^{(\lambda_1-1,\lambda_2-1)}_v P\right),
\end{align}
where $\Phi^{(\lambda_1-1,\lambda_2-1)}_v $ is the Jacobi differential operator applied in the variable $v$. Using the explicit formulas for $\pi_{\lambda_1}^*\otimes\pi_{\lambda_2}^*$ the following result gives the action of the Casimir element. The key point being that the Casimir operator is now an ordinary differential operator of order two only with derivatives in $v$. 
\begin{lemm}
In the new coordinates $(t,v)$, we have
\begin{equation}
\pi_{\lambda_1}^*\otimes\pi_{\lambda_2}^*(C)=\frac{(\lambda_1+\lambda_2)(\lambda_1+\lambda_2-2)}{4}-\Phi^{(\lambda_1-1,\lambda_2-1)}_v.
\end{equation}
\end{lemm}

We now show that spaces of the form $t^\ell \Pol(t)\otimes \C P_\ell^{(\la_1-1,\la_2-1)}$ are irreducible submodules for the tensor product. More precisely we have the following theorem.
\begin{theo}\label{thm:TensorVermaModuleDecomposition}
Let $\lambda_1,\lambda_2\in \C\backslash \Z_{\leq 0}$ such that $\lambda_1+\lambda_2\notin \Z_{\leq 0}$ and $\ell \in \N$.

The space $t^\ell\cdot \Pol(t)\otimes \C P_\ell^{(\lambda_1-1,\lambda_2-1)}(v)$ is an irreducible $\mathfrak{sl}_2$-submodule isomorphic to the lowest weight module $V_{\lambda_1+\lambda_2+2\ell}^*$ with lowest weight vector $t^\ell P_\ell^{(\lambda_1-1,\lambda_2-1)}(v)$.

This gives the following decomposition into irreducible $\mathfrak{sl}_2$-modules.
\begin{equation}
V_{\lambda_1}^*\otimes V_{\lambda_2}^*\simeq \bigoplus_{\ell\geq 0}V_{\lambda_1+\lambda_2+2\ell}^*
\end{equation}

\end{theo}
\begin{proof}
Using our remarks about Jacobi polynomials, one can show by a direct computation that the space $t^\ell\cdot Pol(t)\otimes \C P_\ell^{(\lambda_1-1,\lambda_2-1)}(v)$ is an eigenspace for the Casimir operator for the eigenvalue $(\lambda_1+\lambda_2+2\ell)(\lambda_1+\lambda_2+2\ell-2)$. Note that the conditions on $\lambda_1,\lambda_2$ ensure that these eigenvalues are all different. Thus it is a $\mathfrak{sl}_2$-submodule for the tensor product. This submodule does not have any highest weight vector according to the $E$-action. Moreover, one can check explicitly that $t^\ell P_\ell^{(\lambda_1-1,\lambda_2-1)}(v)$ is the only lowest weight vector if $\lambda_1+\lambda_2\notin \Z_{\leq 0}$. This proves that the submodule is isomorphic to $V_{\lambda_1+\lambda_2+2\ell}^*$. Finally, Lemma \ref{lem:ImagePsi} proves the decomposition.
\end{proof}

A direct consequence of Theorem \ref{thm:TensorVermaModuleDecomposition} is the following corollary.
\begin{coro}\label{coro;Intertwining}
Let $\lambda_1,\lambda_2\in \C\backslash \Z_{\leq 0}$ such that $\lambda_1+\lambda_2\notin \Z_{\leq 0}$ and $\ell \in \N$.

The map $\Phi_\ell^{\lambda_1,\lambda_2}: \Pol(t)\to \Pol(t,v)$ given by the formula:
\begin{equation}
\Phi_\ell^{\lambda_1,\lambda_2} Q(t,v)=t^\ell Q(t)P_\ell^{(\lambda_1-1,\lambda_2-1)}(v).
\end{equation}
defines an intertwining operator between $V^*_{\lambda_1+\lambda_2+2\ell}$ and $V_{\lambda_1}^*\otimes V_{\lambda_2}^*$.

The counterpart of this formula in the variables $(x,y)$ is given by the map $\widetilde{\Phi}_\ell^{\lambda_1,\lambda_2}:\Pol(t)\to \Pol(x,y)$ defined by the formula:
\begin{equation}
\widetilde{\Phi}_\ell^{\lambda_1,\lambda_2} Q(x,y)=(x+y)^\ell P_\ell^{(\lambda_1-1,\lambda_2-1)}\left(\frac{y-x}{x+y}\right)Q(x+y).
\end{equation}
\end{coro}

\section{Intertwining operators and proof of the main theorem}\label{sec:IntertwiningOperator}

We are ready to interpret the RC brackets as intertwining operators for highest weight Verma modules. The intertwining operator between $V_{\lambda_1}\otimes V_{\lambda_2}$ and $V_{\lambda_1+\lambda_2+2\ell}$ is given by the adjoint of the operator $\widetilde{\Phi}^{\lambda_1,\lambda_2}_\ell$ obtained in Corollary \ref{coro;Intertwining}. The adjoint is taken with respect to the Fischer inner product (naturally extended to $\Pol(x,y)$). We show that this adjoint coincides with the RC brackets. 

To compute the adjoint of $\widetilde{\Phi}_\ell^{\lambda_1,\lambda_2}$ we introduce the maps
\begin{align*}
\Delta: P\in \Pol(z)&\mapsto P(x+y)\in \Pol(x,y),\\
 m(Q): P\in \Pol(x,y)&\mapsto Q(x,y)P(x,y)\in \Pol(x,y),
\end{align*}
where $Q\in \Pol(x,y)$.

\begin{lemm}
We have:
\begin{align}
\Delta^\dagger: P\in \Pol(x,y) &\mapsto P(z,z)\in \Pol(z),\\
 m(Q)^\dagger: P\in \Pol(x,y)&\mapsto Q(\partial_x,\partial_y)P(x,y)\in \Pol(x,y).
\end{align}
\end{lemm}

\begin{proof}
First, we have:
\begin{align*}
\langle\Delta (z^{p+q}),x^p y^q\rangle_F &=\sum_{k=0}^{p+q}\binom{p+q}{k}\langle x^k y^{p+q-k}, x^py^q\rangle_F\\
&=(p+q)! \\
&=\langle z^{p+q},z^{p+q}\rangle_F.
\end{align*}
Moreover, $\langle\Delta (z^n),x^p y^q\rangle_F=0=\langle z^n,z^{p+q}\rangle_F$ if $n\neq p+q$, hence  $\Delta^\dagger P=P(z,z)$.

The second equality follows from the fact that the $\dagger$ operation exchanges $x$ and $\partial_x$.
\end{proof}

Finally, notice that:
\[ \widetilde{\Phi}^{\lambda_1,\lambda_2}_\ell=m\left((x+y)^\ell P_\ell^{(\lambda_1-1,\lambda_2-1)}\left(\frac{y-x}{x+y}\right)\right) \circ \Delta,\] 
and that $(x+y)^\ell P_\ell^{(\lambda_1-1,\lambda_2-1)}(\frac{y-x}{x+y})$ is the polynomial defined in \eqref{eq:JacobiTwoVariables}. Thus the adjoint operator of $\widetilde{\Phi}_\ell^{\lambda_1,\lambda_2}$ is given by the RC bracket as stated in the following theorem.

\begin{theo}\label{thm:RCasSBO}
For $\lambda_1,\lambda_2 \in \C\backslash\Z_{\leq 0}$ such that $\lambda_1+\lambda_2 \notin \Z_{\leq 0}$ and $P\in \Pol(x,y)$ we have:
\begin{equation}
\left(\widetilde{\Phi}^{\lambda_1,\lambda_2}_\ell\right)^\dagger P(z)=\sum_{s=0}^\ell (-1)^s\binom{\ell+\lambda_1-1}{\ell-s}\binom{\ell+\lambda_2-1}{s}\frac{\partial^\ell P}{\partial^{s}_x\partial^{\ell-s}_y}(z,z).
\end{equation}
Hence the RC bracket $\RC{\lambda_1}{\lambda_2}{\ell}$ is an intertwining operator between $V_{\lambda_1}\otimes V_{\lambda_2}$ and $V_{\lambda_1+\lambda_2+2\ell}$.
\end{theo}

\subsection{Proof of Theorem \ref{thm:EquationRC}}

From now on, we fix $n\in \N$ and $\lambda_1,\lambda_2,\lambda_3\in \C$ satisfying the condition
\begin{equation}\label{conditionParametre}
\lambda_1,\ \lambda_2,\ \lambda_3 \notin \Z_{\leq 0} \quad \text{and} \quad \lambda_1+\lambda_2,\ \lambda_2+\lambda_3,\ \lambda_1+\lambda_2+\lambda_3 \notin \Z_{\leq 0}.
\end{equation}

The preceding sections have interpreted the RC brackets as intertwining operators for highest weight Verma modules, and therefore as adjoint to the interwining operators for lowest weight Verma modules. Thus it is clear that there must be an equivalent version of Theorem \ref{thm:EquationRC} for these other intertwining operators. This equivalent version reads as the following convolution formula for Jacobi polynomials which is known (see \cite{KVJ98} or \cite{LabrietPoulain23}).
\begin{theo}\label{theo:ConvolutionRacah}
We have, for $k=0,\cdots,n$, and $\lambda_1,\lambda_2,\lambda_3$ satisfying \eqref{conditionParametre}
\begin{multline}\label{eq;ConvolutionRacah}
(x+y+z)^{n-k}(x+y)^k P_{n-k}^{(\lambda_1+\lambda_2+2k-1,\lambda_3-1)}\left(\frac{z-x-y}{x+y+z}\right)P_k^{(\lambda_1-1,\lambda_2-1)}\left(\frac{y-x}{x+y}\right)=\\
\sum_{p=0}^n
U^{\lambda_1,\lambda_2;k}_{\lambda_3;n,p}(x+y+z)^{n-p}(y+z)^p  P_{n-p}^{(\lambda_1-1,\lambda_2+\lambda_3+2p-1)}\left(\frac{y+z-x}{x+y+z}\right)P_p^{(\lambda_2-1,\lambda_3-1)}\left(\frac{z-y}{y+z}\right).
\end{multline}
\end{theo}
Formula \eqref{eq;ConvolutionRacah} translates at once into a formula for the intertwining operators $\widetilde{\Phi}_k^{\lambda_1,\lambda_2}$ given in Corollary \ref{coro;Intertwining}, and we get:
\begin{equation}\label{eq:ConvolutionIdentity}
(\widetilde{\Phi}_k^{\lambda_1,\lambda_2}\otimes Id)\circ\widetilde{\Phi}_{n-k}^{\lambda_1+\lambda_2+2k,\lambda_3}=\sum_{p=0}^n U^{\lambda_1,\lambda_2;k}_{\lambda_3;n,p}(Id\otimes \widetilde{\Phi}_p^{\lambda_2,\lambda_3})\circ\widetilde{\Phi}_{n-p}^{\lambda_1,\lambda_2+\lambda_3+2p}.
\end{equation}
Taking the adjoint of this formula, following Theorem \ref{thm:RCasSBO}, concludes the proof of Theorem \ref{thm:EquationRC}.

\begin{rema}
One also has the inverse identity:
\begin{equation}\label{eq:EquationRCReverse}
\RC{\lambda_1}{\lambda_2+\lambda_3+2p}{n-p}\circ (Id\otimes \RC{\lambda_2}{\lambda_3}{p})=\sum_{k=0}^n \tilde{U}^{\lambda_1,\lambda_2;k}_{\lambda_3;n,p}\RC{\lambda_1+\lambda_2+2k}{\lambda_3}{n-k}\circ (\RC{\lambda_1}{\lambda_2}{k}\otimes Id),
\end{equation}
where
\begin{equation*}
\tilde{U}^{\lambda_1,\lambda_2;k}_{\lambda_3;n,p}=U^{\lambda_3,\lambda_2;p}_{\lambda_1;n,k}=\binom{n}{p}\frac{
(\lambda_2)_p(\lambda_1)_{n-p}(\lambda_1+\lambda_2+\lambda_3+n-1)_k}{(\lambda_1)_k (\lambda_1+\lambda_2+k-1)_k(\lambda_1+\lambda_2+2k)_{n-k}} R_{k,p}.
\end{equation*}
We have noted the obvious symmetry $R_{k,p}=R_{p,k}|_{\lambda_1\leftrightarrow \lambda_3}$. The reverse identity is actually the same as the original one by exchanging $\lambda_1$ and $\lambda_3$ and using the antisymmetry of the RC brackets.
\end{rema}

\begin{rema}
The interpretation of RC brackets as intertwining operators, even though interesting in its own, is not strictly necessary in our proof of Theorem \ref{thm:EquationRC}. Indeed, one can get the result from the convolution identity for Jacobi polynomials in Theorem \ref{theo:ConvolutionRacah} by calculating adjoints without reference to representations of $\mathfrak{sl}_2$. However, seeing the RC brackets as intertwining operators provides a natural explanation of these identities, and allows to bypass Theorem \ref{theo:ConvolutionRacah}.

Indeed, recall that Theorem \ref{theo:ConvolutionRacah} was obtained in \cite{LabrietPoulain23} using certain representations of the Racah algebra, seen as the centraliser of $\mathcal{U}(\mathfrak{sl}_2)$ diagonally embedded in $\mathcal{U}(\mathfrak{sl}_2)^{\otimes 3}$. The representations were on multiplicity spaces in the tensor product, and of course, can also be given on intertwining operators. Hence, one can prove the identities in Theorem \ref{thm:EquationRC} from the interpretation of the RC brackets as intertwining operators, by repeating the reasoning used in \cite{LabrietPoulain23} for Theorem \ref{theo:ConvolutionRacah}. The technical details would be completely similar, so we have decided to appeal directly to Theorem \ref{theo:ConvolutionRacah}.
\end{rema}

\section{Applications}
\subsection{Associativity of the Eholzer product}
For any $\lambda\in\mathbb{Z}_{> 0}$, we let $V_{\lambda}$ denote the space of holomorphic functions in one variable, and we define $V_0=\C$, identified with the space of constant functions. We set
\[V=\bigoplus_{\lambda\in\mathbb{Z}_{\geq 0}} V_{\lambda}\ .\]
For $\la_1,\la_2,n\geq 0$, we consider the RC bracket $\RC{\la_1}{\la_2}{n}$ defined by Formula \eqref{def:RC} as a linear map between the following vector spaces:
\[\RC{\la_1}{\la_2}{n}\ :\ \ V_{\la_1}\otimes V_{\la_2}\to V_{\la_1+\la_2+2n}\ .\]
We take a formal parameter $\hbar$ and we define the following star product:
\begin{equation}\label{def:StarProduct2}
\forall (f_1,f_2)\in V_{\la_1}\times V_{\la_2}\,,\ \ \ \ \ \ f_1\star f_2 =\sum_{n\geq 0} \RC{\la_1}{\la_2}{n}(f_1,f_2)\hbar^n\ .
\end{equation}
The star product takes values in the space $V^{\hbar}=V[[\hbar]]$ of formal power series in $\hbar$ with coefficients in $V$. Varying $\la_1,\la_2$, the star product extends to give a linear map:
\[\star \ :\ \ V^{\hbar}\otimes V^{\hbar} \to V^{\hbar}\ ,\] 
equipping the vector space $V^{\hbar}$ with a product.

As a Corollary of our set of identities for the RC brackets, we obtain a new proof of the associativity of the star product above, relying on a known generating function for the Racah polynomials.
\begin{coro}\label{coro:Associativity}
The map $\star$ equips $V^{\hbar}$ with an associative product.
\end{coro}

\begin{rema} 
The unit for the star product is the constant function $1$ belonging to $V_0$. Note also that putting $\hbar=0$ recovers the standard product of functions, since $\RC{\la_1}{\la_2}{0}(f,g)=fg$. 

Consider the subspace obtained by taking in each $V_{\lambda}$ the subspace of modular forms of weight $\lambda$. By the classical property of the RC bracket, the star product restricts to this subspace, and is called the Eholzer product. We refer to \cite{CMZ97}.
\end{rema}

\begin{proof}
The associativity condition can be rewritten as follows for all $n\in \N$:
\begin{equation}
\sum_{k=0}^n \RC{\lambda_1+\lambda_2+2k}{\lambda_3}{n-k}\circ (\RC{\lambda_1}{\lambda_2}{k}\otimes Id)=\sum_{p=0}^n \RC{\lambda_1}{\lambda_2+\lambda_3+2p}{n-p}\circ (Id\otimes \RC{\lambda_2}{\lambda_3}{p}). 
\end{equation}
Using Theorem \ref{thm:EquationRC} we get
\begin{equation}
\sum_{k=0}^n \sum_{p=0}^n U^{\lambda_1,\lambda_2;k}_{\lambda_3;n,p}\RC{\lambda_1}{\lambda_2+\lambda_3+2p}{n-p}\circ (Id\otimes \RC{\lambda_2}{\lambda_3}{p})=\sum_{p=0}^n \RC{\lambda_1}{\lambda_2+\lambda_3+2p}{n-p}\circ (Id\otimes \RC{\lambda_2}{\lambda_3}{p}). 
\end{equation}
The associativity is then equivalent to 
\begin{equation}
\sum_{k=0}^n U^{\lambda_1,\lambda_2;k}_{\lambda_3;n,p}=1\ \ \ \ \ \text{for any $p=0,1,\dots,n$.} 
\end{equation}
This follows from a more general statement giving the generating function for the coefficients $U^{\lambda_1,\lambda_2;k}_{\lambda_3;n,p}$, which itself follows from straightforward manipulations using a known generating function for the Racah polynomials. This is shown in Proposition \ref{prop:genseries} below.
\end{proof}

\begin{prop}\label{prop:genseries}
Let $p\in\{0,1,\dots,n\}$. The generating function of the coefficients $U^{\lambda_1,\lambda_2;k}_{\lambda_3;n,p}$ is:
\begin{multline}
\sum_{k=0}^n U^{\lambda_1,\lambda_2;k}_{\lambda_3;n,p}t^k=\frac{(\lambda_3)_n(\lambda_1+\lambda_2+\lambda_3+n-1)_p}{(\lambda_3)_p(\lambda_2+\lambda_3+p-1)_p(\lambda_2+\lambda_3+2p)_{n-p}}\times \\
\times {}_2F_1\left(\begin{array}{c}-p\ ,\ \lambda_1+n-p\\ \lambda_{1}+\lambda_2+\lambda_3+n-1\end{array};t\right){}_2F_1\left(\begin{array}{c} p-n\ ,\ p+\lambda_2\\ -\lambda_3-n+1\end{array};t\right)\ . 
\end{multline}
and in particular, we have:
\begin{equation}
\sum_{k=0}^n U^{\lambda_1,\lambda_2;k}_{\lambda_3;n,p}=1\ \ \ \ \ \text{for any $p=0,1,\dots,n$.}
\end{equation}
\end{prop}
\begin{proof}
We use the following properties involving Pochammer symbols 
\[
\binom{n}{k}=(-1)^k\frac{(-n)_k}{k!} \quad\text{ and }\quad (x)_{n-k}=(-1)^{k}\frac{(x)_n}{(-x-n+1)_{k}}\ .
\]
Hence using the formula \eqref{eq:FormulaCoeff} for $U^{\lambda_1,\lambda_2;k}_{\lambda_3;n,p}$ we get 
\begin{equation*}
\sum_{k=0}^n U^{\lambda_1,\lambda_2;k}_{\lambda_3;n,p}t^k=\frac{(\lambda_3)_n(\lambda_1+\lambda_2+\lambda_3+n-1)_p}{(\lambda_3)_p(\lambda_2+\lambda_3+p-1)_p(\lambda_2+\lambda_3+2p)_{n-p}}
\sum_{k=0}^n \frac{(-n)_k(\lambda_2)_k}{(-n-\lambda_3+1)_k k!}R_{p,k}t^k.
\end{equation*}
The sum in the right hand side is precisely one of the known generating function for Racah polynomials. We refer to \cite{KLS10} and indicate that the correspondence with the parameters we use is $\gamma+1=-n$, $\alpha+1=\lambda_2$, $\beta+1=\lambda_3$ and $\delta=\lambda_1+\lambda_3+n-1$. The second generating function given in \cite{KLS10} is, using the obvious symmetry $R_{k,p}=R_{p,k}|_{\lambda_1\leftrightarrow \lambda_3}$,
\begin{equation*}
{}_2F_1\left(\begin{array}{c}-p\ ,\ \lambda_1+n-p\\ \lambda_1+\lambda_2+\lambda_3+n-1\end{array};t\right){}_2F_1\left(\begin{array}{c} p-n\ ,\ p+\lambda_2\\ -\lambda_3-n+1\end{array};t\right)
=\sum_{k=0}^n \frac{(-n)_k(\lambda_2)_k}{(-\lambda_3-n+1)_k k!}R_{p,k}t^k\ .
\end{equation*}
This proves the statement about the generating series. 

Now we evaluate it at $t=1$. Note that since $p\leq n$, both $_2F_1$ are actually polynomials. We use the following evaluation of hypergeometric series
\[{}_2F_1\left(\begin{array}{c}-p\ ,\ b\\ c\end{array};1\right)=\frac{(c-b)_p}{(c)_p}\ .\]
Hence we have:
\begin{multline*}
{}_2F_1\left(\begin{array}{c}-p\ ,\ \lambda_1+n-p\\ \lambda_1+\lambda_2+\lambda_3+n-1\end{array};1\right){}_2F_1\left(\begin{array}{c} p-n\ ,\ p+\lambda_2\\ -\lambda_3-n+1\end{array};1\right)\\
=\frac{(\lambda_2+\lambda_3+p-1)_p(-n-\lambda_3-p-\lambda_2+1)_{n-p}}{(\lambda_1+\lambda_2+\lambda_3+n-1)_p(-n-\lambda_3+1)_{n-p}}.
\end{multline*}
This concludes the proof since $(-n-\lambda_3-\lambda_2-p+1)_{n-p}=(-1)^{n-p}(\lambda_2+\lambda_3+2p)_{n-p}$ and $(-n-\lambda_3+1)_{n-p}=(-1)^{n-p}(\lambda_3+p)_{n-p}=(-1)^{n-p}\frac{(\lambda_3)_n}{(\lambda_3)_p}$.
\end{proof}

\begin{rema}
In \cite{CMZ97}, the authors proved the associativity for a one parameter family of star products: 
\begin{equation}
f\star_\kappa g =\sum_n~ t_n^\kappa(\lambda_1,\lambda_2)[f,g]_n \hbar^n,
\end{equation}
with
\[
t_n^\kappa(\lambda_1,\lambda_2)=\frac{1}{\binom{-2\lambda_2}{n}}\sum_{r+s=n}\frac{\binom{-\lambda_1}{r}\binom{-\lambda_1+\kappa-1}{r}\binom{n+\lambda_1+\lambda_2-\kappa}{s}\binom{n+\lambda_1+\lambda_2+-1}{s}}{\binom{-2\lambda_1}{r}\binom{2n+2\lambda_1+2\lambda_2-2}{s}},
\]
which are conjectured to be equal to
\[
t_n^\kappa(\lambda_1,\lambda_2)=\left(-\frac{1}{4}\right)^n\sum_{j\geq 0}\binom{n}{2j}\frac{\binom{-\frac{1}{2}}{j}\binom{\kappa-\frac{3}{2}}{j}\binom{\frac{1}{2}-\kappa}{j}}{\binom{-\lambda_1-\frac{1}{2}}{j}\binom{-\lambda_2-\frac{1}{2}}{j}\binom{n+\lambda_1+\lambda_2-\frac{3}{2}}{j}}.
\]
The case $\kappa=\frac{1}{2}$ or $\frac{3}{2}$ corresponds to the Eholzer product \eqref{def:StarProduct}. The associativity for this family of star products is then equivalent to the following identities for the Racah coefficients:
\begin{equation*}
\sum_{k=0}^n U^{\lambda_1,\lambda_2;k}_{\lambda_3;n,p}t_k^\kappa(\lambda_1,\lambda_2)t_{n-k}^\kappa(\lambda_1+\lambda_2+2k,\lambda_3)=t_p^\kappa(\lambda_2,\lambda_3)t_{n-k}^\kappa(\lambda_1,\lambda_2+\lambda_3+2p)\ .
\end{equation*}
Thus the associativity of the star products provide (as far as we know) new identities for the Racah coefficients. On the other hand, if we were able to prove these identities directly, this would provide another proof of the associativity of the above family of star products.
\end{rema}

\subsection{About Rankin--Cohen algebras}

A second application of Theorem \ref{thm:EquationRC} is related to the definition of Rankin--Cohen algebras. In the article \cite{Zagier94}, D.Zagier introduced the definition of a \emph{Rankin--Cohen algebra} (or \emph{RC algebra}) as follows:
\begin{defi}[\cite{Zagier94}]
A Rankin--Cohen algebra over a field $K$ is a graded vector space $M_*=\bigoplus_{k\in \N} M_k$ with $M_0=K\cdot 1$ and $\dim M_k$ finite, together with bilinear operations $[\cdot,\cdot]_n:M_k\otimes M_\ell \to M_{k+\ell+2n}$  which satisfy all the algebraic identities satisfied by the RC brackets.
\end{defi}

A first algebraic identity satisfied by the RC brackets is:
\begin{equation}\label{eq:CommutativityRC}
[f,g]_n=(-1)^n[g,f]_n\ .
\end{equation}
Other examples of algebraic identities are \eqref{eq:IdentityRankinCohen1} and \eqref{eq:IdentityRankinCohen2}, and a large class of examples is given in \cite{Zagier94}. However, the question is asked in \cite{Zagier94} of what would be a complete set for these algebraic identities, namely, a set of identities implying all the others.

By an algebraic identity satisfied by the RC brackets, we understand the following. 
We use the notations of the preceding subsection. Let $D\geq 0$ and $\lambda_1,\dots,\lambda_{D+1}\in\mathbb{Z}_{\geq 0}$. Take holomorphic functions $f_i\in V_{\la_i}$ and fix $n\geq 0$. We consider all possible ways to compose RC brackets in order to create a map:
\[V_{\la_1}\otimes\dots\otimes V_{\la_{D+1}}\to V_{\la_1+\dots+\la_{D+1}+2n}\ ,\]
that is, to create a function in $V_{\la_1+\dots+\la_{D+1}+2n}$ out of $(f_1,\dots,f_{D+1})$ using RC brackets. More precisely, we consider all possible bracketings of the expressions $f_{i_1} f_{i_2}\dots f_{i_{D+1}}$, for all possible permutations $i_1,\dots,i_{D+1}$ of $1,\dots,D+1$. For example, a standard one is:
\begin{equation}\label{bracketing}
[f_1,[f_2,\cdots\cdots,[f_D,f_{D+1}]_{k_1}]_{k_2}]\dots ]_{k_D}\ \ \ \text{ where } k_1+\dots+k_D=n\,.
\end{equation}
Finally, an algebraic identity is a linear combination of such operations giving $0$ (for all functions $f_1,\dots,f_{D+1}$). Note that the coefficients of the linear combination will depend on the weights $\la_i$'s, as in the example \eqref{eq:IdentityRankinCohen2}.

Theorem \ref{thm:EquationRC} gives a large set of such identities. They read: 
\begin{equation}\label{eq:EquationRCbrackets2}
[[f_1,f_2]_k,f_3]_{n-k}=\sum_{p=0}^nU^{\lambda_1,\lambda_2;k}_{\lambda_3;n,p}[f_1, [f_2,f_3]_p]_{n-p}\ ,
\end{equation}
and thus allows to rewrite a bracket of the form $[[f_1,f_2]_k,f_3]_{n-k}$ in terms of the standard ones \eqref{bracketing}, with $D=2$.

Now let $D\geq 1$ and consider an arbitrary bracketing of $f_{i_1} f_{i_2}\dots f_{i_{D+1}}$ for an arbitrary permutation $i_1,\dots,i_{D+1}$ of $1,\dots,D+1$. It is straightforward to see that using repeatedly the identities \eqref{eq:CommutativityRC} and \eqref{eq:EquationRCbrackets2}, which play the roles, respectively, of commutativity and associativity for the RC brackets, one can write an arbitrary bracketing as a linear combination of the standard ones \eqref{bracketing}.

Moreover, the standard bracketings in \eqref{bracketing}, when $k_1,\dots,k_D$ vary, are linearly independent. It follows by direct inspection of the degrees of the partial derivatives of the functions $f_1,\dots,f_{D+1}$. It also follows from their interpretation as intertwining operators from $V_{\la_1}\otimes \dots \otimes V_{\la_{D+1}}$ to the different copies of the Verma module $V_{\la_1+\dots+\la_{D+1}+2n}$.

Thus, for every linear combination of compositions of brackets, using only \eqref{eq:CommutativityRC} and \eqref{eq:EquationRCbrackets2}, we can write it as a linear combination of the linearly independent standard ones, and therefore, if the original linear combination is an algebraic identity, we find it to be $0$ by this procedure. We conclude that we have proved the following statement.
\begin{prop}
All algebraic identities for the RC brackets are consequences of \eqref{eq:CommutativityRC} and the identities of Theorem \ref{thm:EquationRC}. 
\end{prop}
Note that the given set of identities is minimal in the sense that one cannot obtain one of these identities from the others. 

\begin{rema}
The different bracketings of the set of functions $f_1,\dots,f_{D+1}$ are in correspondence with the so-called binary coupling schemes for performing the tensor product of $D+1$ representations of $\mathfrak{sl}_2$. The fact that all identities between different bracketings are obtained applying repeatedly the ones for $D=2$ is reminiscent of the fact that we can move between different coupling schemes using repeatedly the relations between coupling schemes for $D=2$. In terms of orthogonal polynomials, the polynomials obtained as transition coefficients are expressed in terms of products of univariate Racah polynomials. See for example \cite{CFR23,FLVJ02,LabrietPoulain23,VdJ03} for more details. 
\end{rema}

\begin{example}
Let us prove identity \eqref{eq:IdentityRankinCohen1} as an example. To do so we expand each components of the sum in the basis $\{[f_1,[f_2,f_3]_0]_2,[f_1,[f_2,f_3]_1]_1,[f_1,[f_2,f_3]_2]_0\}$ which gives:
\begin{align*}
[[f_1,f_2]_1,f_3]_1&=U^{\lambda_1,\lambda_2;1}_{\lambda_3;2,0}[f_1,[f_2,f_3]_0]_2+U^{\lambda_1,\lambda_2;1}_{\lambda_3;2,1}[f_1,[f_2,f_3]_1]_1+U^{\lambda_1,\lambda_2;1}_{\lambda_3;2,2}[f_1,[f_2,f_3]_2]_0,\\
[[f_2,f_3]_1,f_1]_1&=-[f_1,[f_2,f_3]_1]_1,\\
[[f_3,f_1]_1,f_2]_1&=-U^{\lambda_1,\lambda_3;1}_{\lambda_2;2,0}[f_1,[f_2,f_3]_0]_2+U^{\lambda_1,\lambda_3;1}_{\lambda_2;2,1}[f_1,[f_2,f_3]_1]_1-U^{\lambda_1,\lambda_3;1}_{\lambda_2;2,2}[f_1,[f_2,f_3]_2]_0.
\end{align*}
Using formula \eqref{eq:FormulaCoeff}, one gets:
\begin{align*}
U^{\lambda_1,\lambda_2;1}_{\lambda_3;2,0}&=\frac{2\lambda_2\lambda_3}{\lambda_2+\lambda_3},\\
U^{\lambda_1,\lambda_2;1}_{\lambda_3;2,1}&=\frac{\lambda_1\lambda_2+\lambda_2\lambda_3-\lambda_3\lambda_1+2\lambda_2+\lambda_2^2}{(\lambda_2+\lambda_3)(\lambda_2+\lambda_3+2)}, \\
U^{\lambda_1,\lambda_2;1}_{\lambda_3;2,2}&=-\frac{2\lambda_1(\lambda_1+\lambda_2+\lambda_3+2)}{(\lambda_2+\lambda_3+1)(\lambda_2+\lambda_3+2)(\lambda_2+\lambda_3+4)}.
\end{align*}
Now a direct computation shows formula \eqref{eq:IdentityRankinCohen1}.

An example involving more brackets is given by the following identity:
\[
[[[f_1,f_2]_0,f_3]_0,f_4]_1+[[[f_2,f_3]_0,f_4]_0,f_1]_1+[[[f_4,f_3]_0,f_1]_0,f_2]_1+[[[f_4,f_1]_0,f_2]_0,f_3]_1=0,
\]
which is proved similarly using an expansion in the basis $\{[f_1,[f_2,[f_3,f_4]_0]_0]_1,[f_1,[f_2,[f_3,f_4]_0]_1]_0,$ $ [f_1,[f_2,[f_3,f_4]_1]_0]_0\}$.
\end{example}

\begin{rema}
In \cite[Corollary p.65]{Zagier94}, D.Zagier mentioned a way to obtain universal identities for the RC brackets. It was proved that the following  expression is invariant under the permutations of the triplets $(f_1,\lambda_1,x),(f_2,\lambda_2,y),(f_3,\lambda_3,z)$:
\[
\sum_{k=0}^n c_k(\lambda_1,\lambda_2;x,y)c_{n-k}(\lambda_1+\lambda_2+2k,\lambda_3;x+y,z)[[f_1,f_2]_k,f_3]_{n-k},
\]
with $c_k(\lambda_1,\lambda_2;x,y)=(2k+\lambda_1+\lambda_2-1)\frac{k!(k+\lambda_1+\lambda_2-2)!}{(k+\lambda_1-1)!(k+\lambda_2-1)!}(x+y)^nP_n^{\lambda_1-1,\lambda_2-1}\left(\frac{y-x}{x+y}\right)$. Varying $n$ and looking at the coefficients of the various monomials in $x,y,z$, one obtains identities for the RC brackets. 

One can easily recover the invariance of the above expressions using Theorem \ref{thm:EquationRC} and Theorem \ref{theo:ConvolutionRacah}, together with the formula \eqref{eq:CommutativityRC}. Thus the algebraic identities \cite[Corollary p.65]{Zagier94} can also be obtained as consequences of the results presented in this paper.
\end{rema}

\end{document}